\documentclass[10pt,a4paper]{amsart}
\usepackage[T1]{fontenc}
\usepackage{lmodern}
\usepackage{amsmath,amssymb,amsthm,mathtools}
\usepackage{booktabs,tabularx,array,multirow}
\usepackage[margin=26mm]{geometry}
\usepackage{microtype}
\usepackage{xcolor}
\usepackage{enumitem}

\usepackage[pagebackref=false,colorlinks=true,linkcolor=blue!45!black,citecolor=blue!45!black,urlcolor=blue!45!black]{hyperref}

\newtheorem{theorem}{Theorem}[section]
\newtheorem{lemma}[theorem]{Lemma}
\newtheorem{proposition}[theorem]{Proposition}
\newtheorem{corollary}[theorem]{Corollary}

\theoremstyle{remark}
\newtheorem{remark}[theorem]{Remark}

\newtheorem{mainthm}{\bf Theorem}

\DeclareMathOperator{\Aut}{Aut}
\DeclareMathOperator{\Out}{Out}
\DeclareMathOperator{\Syl}{Syl}
\DeclareMathOperator{\soc}{soc}
\DeclareMathOperator{\GL}{GL}

\DeclareMathOperator{\Sp}{Sp}

\DeclareMathOperator{\PSU}{PSU}
\DeclareMathOperator{\PGL}{PGL}

\DeclareMathOperator{\Inn}{Inn}
\DeclareMathOperator{\Inndiag}{Inndiag}
\DeclareMathOperator{\SD}{\mathsf{SD}}

\newcommand{\F}{\mathbb F}

\newcommand{\Cent}[2]{\mathrm{C}_{#1}(#2)}
\newcommand{\Norm}[2]{\mathrm{N}_{#1}(#2)}

\newcommand{\OddCore}[1]{\mathrm{O}_{2'}(#1)}

\newcommand{\OmegaOne}[1]{\Omega_1(#1)}
\makeatletter
\newtoks\address@list
\newcommand{\addaddress}[3]{%
  \address@list=\expandafter{\the\address@list
    \textsc{#1, #2.} \par
    \textit{Email address}: \href{mailto:#3}{\textsf{#3}} \par
    \addvspace{\medskipamount}%
  }%
}
\newcommand{\printaddresses}{%
  \AtEndDocument{\bigskip{\footnotesize
    \the\address@list
  }}%
}
\makeatother

\title[Brauer's Problem 33]{Sylow $2$-Subgroups of Finite Groups with a Single Conjugacy Class of Involutions}

\author{Tianze Li and Yu Zeng}

\addaddress{Tianze Li}{School of Mathematical Sciences, Huaqiao University, Quanzhou, Fujian, 362021, China}{tzli@hqu.edu.cn}
\addaddress{Yu Zeng (Corresponding Author)}{Department of Mathematics, Suzhou University of Technology, Changshu, Jiangsu, 215500, China}{yuzeng2004@163.com}

\printaddresses  

\date{}
\subjclass[2020]{20D20, 20D05, 20D06, 20E45}
\keywords{Involution, Sylow $2$-subgroup, almost simple group, Suzuki $2$-group, Brauer's Problem 33}
\hypersetup{
  pdftitle={Brauer's Problem 33},
  pdfsubject={Finite groups with one conjugacy class of involutions},
  pdfkeywords={involution, Sylow 2-subgroup, almost simple group, Suzuki 2-group}
}
\begin{document}
\raggedbottom
\begin{abstract}
We solve Brauer's Problem~33 by determining the finite $2$-groups that
occur as Sylow $2$-subgroups of finite groups with a single conjugacy
class of involutions. Besides the Sylow $2$-subgroups arising from the
relevant non-abelian simple groups in Table \ref{tab:simple-list}, these are precisely the homocyclic,
generalized quaternion, and Suzuki $2$-groups. 
The classification of finite simple groups is used in this paper.
\end{abstract}
\maketitle

\section{Introduction}

In 1963, Brauer posed the following problem as Problem~33 in
\cite{Brauer}: 
determine the finite $2$-groups that occur as Sylow
$2$-subgroups of finite groups with exactly one conjugacy class of
involutions. The problem asks how a strong global fusion condition
restricts the local structure of a Sylow $2$-subgroup.

Early work proceeded in two related directions. Shaw \cite{Shaw} studied
Sylow $2$-subgroups of solvable groups with one conjugacy class of
involutions, while Higman's work on Suzuki $2$-groups \cite{Higman}
initiated the study of non-abelian $2$-groups admitting automorphisms that
act transitively on their involutions. Subsequent work of Shult
\cite{Shult}, Gross \cite{Gross}, Bryukhanova \cite{Bryukhanova}, and
Wilkens \cite{Wilkens} developed this automorphism-theoretic approach. In
particular, Gross conjectured that a finite $2$-group whose automorphism
group is transitive on its involutions must be homocyclic, generalized
quaternion, or a Suzuki $2$-group. Li and Zhu recently proved this
conjecture as part of their classification of finite $p$-groups whose
elements of order $p$ form a single fusion class
\cite[Theorem~1.1]{LZ}. On the non-solvable side, the
Alperin--Brauer--Gorenstein analysis of groups with semidihedral Sylow
$2$-subgroups \cite{ABG} is relevant to the exceptional almost simple
family that arises below.

The connection with Brauer's problem is natural but incomplete. If $P$
is a Sylow $2$-subgroup of a finite group whose involutions form one conjugacy
class, then the involutions of $P$ are conjugate in the ambient group.
The conjugating elements, however, need not normalize $P$, so this
conjugacy need not be induced by automorphisms of $P$. Thus the theorem
of Li and Zhu does not apply until this fusion gap has been closed. The
purpose of this paper is to close that gap and to combine the resulting
reduction with the analysis of almost simple groups.

\begin{mainthm}\label{thm:brauer}
Let $P$ be a nontrivial finite $2$-group. Then $P$ is a Sylow
$2$-subgroup of a finite group whose involutions form one conjugacy class
if and only if at least one of the following holds:
\begin{enumerate}[label=\textup{(\roman*)}]
\item $P$ is homocyclic;
\item $P$ is generalized quaternion;
\item $P$ is a Suzuki $2$-group;
\item $P$ is isomorphic to a Sylow $2$-subgroup of one of the non-abelian
simple groups in Table~\ref{tab:simple-list}.
\end{enumerate}
\end{mainthm}

The proof begins by factoring out the largest normal subgroup of odd
order and showing that the resulting group has a unique minimal normal
subgroup. If this subgroup is elementary abelian, a normalizer argument,
Hering's classification of transitive linear groups, and Alperin's
rank-three theorem close the fusion gap, after which the theorem of Li
and Zhu applies. If the minimal normal subgroup is non-abelian simple, the
ambient group is almost simple. The analysis then reduces to the
involution classes and outer automorphisms of finite simple groups. The
only additional even-index ambient groups are mixed diagonal--field
extensions of $\mathrm{PSL}_2(p^{2a})$ with $p>2$; their Sylow $2$-subgroups are
semidihedral and already occur in $\mathrm{PSL}_3(p^a)$ or
$\PSU_3(p^a)$. The converse follows from explicit realizations of the
three automorphism-theoretic families and from the simple groups in
Table~\ref{tab:simple-list}.

The paper is organized as follows. Section 2 gives the preliminary
reductions and identifies the fusion gap. Section 3 completes
the abelian-socle case by establishing the required
linear-group restriction and eliminating its three-dimensional exception.
Section 4 treats almost simple groups, first at odd outer index
and then at even outer index. The final section combines the two
branches to prove
Theorem~\ref{thm:brauer}.

\section*{AI declaration}
With Gross' conjecture now resolved and assuming the Classification of Finite Simple Groups (CFSG), the line of research on Brauer's Problem 33 has become fairly well-understood.

The authors used GPT-6 as an aid in the work reported in this paper.
In the proof of Lemma~\ref{lem:linear}, GPT-6 helped identify the
restriction of the natural six-dimensional module for $G_2(q)$,
in characteristic $2$, to its subgroup $\mathrm{SL}_3(q)$. GPT-6 also helped locate the relevant
reference for the split Cayley algebra model and the embedding
$\mathrm{SL}_3(q)\hookrightarrow G_2(q)$; see
\cite[Definition~2.1.9, Example~2.1.12, and Theorems~2.2.1 and~2.2.6]
{AHW}. The authors checked the resulting argument and take
responsibility for the mathematical content of the paper.

\enlargethispage{2\baselineskip}

\section{Preliminaries}

\subsection{Notation}

We mainly follow the notation from \cite{Huppert} for group theory and \cite{GLS} for finite simple
groups. 
Throughout, all groups are finite, and $p$ denotes a prime.
For a positive integer $n$ and
a prime $p$, we write $n_p$ to denote the $p$-\emph{part} of $n$, i.e., the maximal $p$-power divisor of $n$.

Let $G$ be a finite group.  
We use $G^{\#}$ to denote the set of nontrivial elements of $G$;
$\mathrm{soc}(G)$ to denote the \emph{socle} of $G$;
$\mathcal{I}(G)$ to denote the set of involutions in $G$.
We write $G\in \mathbf{OC}_2$ if 
$\mathcal{I}(G)$ is nonempty and forms a single $G$-conjugacy class.
For $E\unlhd G$, the group of automorphisms of $E$ induced by
conjugation in $G$ is
\[
 \Aut_G(E):=\Norm{G}{E}/\Cent{G}{E}=G/\Cent{G}{E}.
\]
For a positive integer $t$, let $G^t$ denote the direct product of
$t$ copies of $G$, with convention $G^0=1$. Similarly, for a vector space $V$
over a field $\mathbb{F}$, let $V^t$ denote the direct sum of $t$
copies of $V$, with convention $V^0=0$.

For specific groups, we write $\mathsf{C}_n$ for a cyclic group of
order $n$; $\mathsf{ES}(3_+^3)$ for the extraspecial group of
order $3^3$ and exponent $3$;
$\mathsf{A}_n$ and $\mathsf{S}_n$ for the alternating and symmetric groups of degree $n$, respectively;
$\mathsf{D}_{2^n}$ ($n\geq 3$), $\mathsf{Q}_{2^n}$ ($n\geq 3$), and $\mathsf{SD}_{2^n}$ ($n\geq 4$)
for the dihedral, generalized quaternion, and semidihedral groups of
order $2^n$, respectively.

Other notation will be introduced as needed.

\subsection{Auxiliary results}

We begin by factoring out the odd core and determining the possible
socles of the resulting groups.

\begin{lemma}\label{lem:oddcore}
Let $G$ be a finite group, and set $N=\mathrm{O}_{2'}(G)$. Then $G$ has
exactly one conjugacy class of involutions if and only if $G/N$ does.
\end{lemma}
\begin{proof}
Set $\overline{G}=G/N$, and let $\overline{x}\in\overline{G}$ be an
involution. The preimage of $\langle\overline{x}\rangle$ in $G$ is an
extension of the odd-order group $N$ by a group of order $2$. By the
Schur--Zassenhaus theorem, $N$ has a complement of order $2$ in this
preimage. Thus every involution of $\overline{G}$ lifts to an involution
of $G$. It follows that if $\mathcal{I}(G)$ is a single $G$-conjugacy
class, then $\mathcal{I}(\overline{G})$ is a single
$\overline{G}$-conjugacy class.

Conversely, suppose that $\mathcal{I}(\overline{G})$ is a single
$\overline{G}$-conjugacy class, and let $x,y\in\mathcal{I}(G)$.
After conjugating $x$ by an element of $G$, we may assume that
$\overline{x}=\overline{y}$. Then
\[
N\langle x\rangle=N\langle y\rangle,
\]
and $\langle x\rangle$ and $\langle y\rangle$ are complements to $N$ in
this group. The conjugacy part of the Schur--Zassenhaus theorem shows
that $x$ and $y$ are $N$-conjugate. Hence $\mathcal{I}(G)$ is a single
$G$-conjugacy class.
\end{proof}

\begin{lemma}\label{lem:monolithic}
Let $G$ be a finite group with $\mathrm{O}_{2'}(G)=1$. 
Suppose that
$\mathcal{I}(G)$ is a single $G$-conjugacy class. Then $G$ has a unique
minimal normal subgroup $N=\soc(G)$, and one of the following holds:
\begin{enumerate}[label=\textup{(\roman*)}]
\item $N\cong \mathsf{C}_2^n$ and $\mathcal{I}(G)=N^\#$. In particular,
$G$ acts transitively on $N^\#$;
\item $G$ is almost simple with socle $N$, and
$\mathcal{I}(G)=\mathcal{I}(N)$.
\end{enumerate}
\end{lemma}
\begin{proof}
Let $N$ be a minimal normal subgroup of $G$. Since
$\mathrm{O}_{2'}(G)=1$, the group $N$ has even order and hence contains
an involution. As $\mathcal{I}(G)$ is a single $G$-conjugacy class and
$N$ is normal in $G$, we have $\mathcal{I}(G)\subseteq N^\#$ and thus
$\mathcal{I}(G)=\mathcal{I}(N)$. Any two minimal normal subgroups
therefore contain all involutions of $G$. Since distinct minimal normal
subgroups intersect trivially, $N$ is unique. Hence $N=\soc(G)$.

If $N$ is solvable, then $N\cong\mathsf{C}_2^n$, so
$\mathcal{I}(G)=N^\#$ and (i) holds. Suppose that $N$ is nonsolvable.
Then $N=S_1\times\cdots\times S_t$, where the $S_i$ are isomorphic
non-abelian simple groups. If $t>1$, choose involutions $a\in S_1$ and
$b\in S_2$. Conjugation by $G$ permutes the factors $S_i$, whereas $a$
and $ab$ have nonidentity components in different numbers of factors.
They therefore cannot be $G$-conjugate. Thus $t=1$, and $N$ is
non-abelian simple. Consequently, $G$ is almost simple with socle $N$,
as in (ii).
\end{proof}

\begin{lemma}\label{lem:fusion-gap}
Let $G \in \mathbf{OC}_2$, and let $P \in \mathrm{Syl}_2(G)$. 
Suppose that $E:=\mathrm{soc}(G)$ is solvable and $\mathrm{O}_{2'}(G)=1$.
 The following conditions are equivalent:
\begin{enumerate}[label=\textup{(\roman*)}]
\item $\Aut(P)$ is transitive on $\mathcal{I}(P)$;
\item $E\leq \mathrm{Z}(P)$;
\item $P\leq \Cent{G}{E}$.
\end{enumerate}
If these conditions hold, then $\Norm{G}{P}$ acts transitively on
$\mathcal{I}(P)$ by conjugation.
\end{lemma}
\begin{proof}
By Lemma~\ref{lem:monolithic}, we have $\mathcal{I}(P)=E^\#$.
Suppose that (i) holds, and choose an involution in $\mathrm{Z}(P)$.
Since $\mathrm{Z}(P)$ is characteristic in $P$, transitivity of
$\mathrm{Aut}(P)$ on $\mathcal{I}(P)$ implies that every involution of
$P$ lies in $\mathrm{Z}(P)$. Thus (ii) holds. Conditions (ii) and (iii)
both assert that $P$ centralizes $E$, so they are equivalent.

Assume (ii), and let $x,y\in E^\#$. Choose $g\in G$ such that $x^g=y$.
Both $P^g$ and $P$ are Sylow $2$-subgroups of $\mathrm{C}_G(y)$.
By Sylow conjugacy, there exists $c\in\mathrm{C}_G(y)$ such that
\[
(P^g)^c=P.
\]
Then $gc\in\Norm{G}{P}$ and $x^{gc}=y$. 
So $\Norm{G}{P}$ is
transitive on $\mathcal{I}(P)$, proving the final assertion and (i).
\end{proof}

Lemma~\ref{lem:fusion-gap} shows that applying the theorem of Li and Zhu
requires proving that a Sylow $2$-subgroup centralizes the socle.
In the proof of Theorem~\ref{thm:solvable}, we apply the following
normalizer lemma inside the normalizer of a Sylow $2$-subgroup
of the centralizer of the socle.

For a $2$-group $P$, we write $\OmegaOne{P}$ for the subgroup
generated by the involutions of $P$.

\begin{lemma}\label{lem:normalizer}
Let $G$ be a finite group, and let $N \unlhd G$ be a normal subgroup admitting a normal Sylow $2$-subgroup $A$. 
Let $E = \Omega_1(A)$, and suppose that $N \leq \mathrm{C}_G(E)$ and $\mathcal{I}(G) \subseteq E$. 
Set $H = G/N$. 
If $X \leq H$ has odd order and $\mathrm{C}_E(X) = 1$, then $\mathrm{N}_H(X)$ has odd order.
\end{lemma}
\begin{proof}
Let $L\leq G$ be the preimage of $X$, so that $L/N=X$.
Since $A$ is a $2$-group and $|L/A|=|X||N/A|$ is odd, the
Schur--Zassenhaus theorem gives $L=A\rtimes Y$ for a Hall
$2'$-subgroup $Y$ of $L$. Moreover, all Hall $2'$-subgroups of $L$
are conjugate.
We first show that $\mathrm{C}_A(Y)=1$. Otherwise,
$\mathrm{C}_A(Y)$ contains an involution $a$. Since $E=\Omega_1(A)$,
we have $a\in E$, and hence
$a\in\mathrm{C}_E(Y)=\mathrm{C}_E(X)=1$, a contradiction.

Applying the Frattini argument to $L\unlhd\mathrm{N}_G(L)$ and its
Hall $2'$-subgroup $Y$, we obtain
\[
\mathrm{N}_G(L)
 =L\,\mathrm{N}_{\mathrm{N}_G(L)}(Y)
 =A\bigl(\mathrm{N}_G(Y)\cap\mathrm{N}_G(L)\bigr)
 \leq N\,\mathrm{N}_G(Y).
\]
Conversely, both $N$ and $\mathrm{N}_G(Y)$ normalize $NY=L$, so
$N\mathrm{N}_{G}(Y)\leq \mathrm{N}_{G}(L)$. Therefore
\[
\mathrm{N}_H(X)=\mathrm{N}_{G/N}(L/N)= \mathrm{N}_{G}(L)/N
 =\mathrm{N}_G(Y)N/N\cong \mathrm{N}_{G}(Y)/\mathrm{N}_{N}(Y).
\]
Moreover,
$\mathrm{N}_G(Y)\cap A
 =\mathrm{N}_A(Y)
 =\mathrm{C}_A(Y)
 =1$.
Thus $\mathrm{N}_G(Y)\cap E=1$. Since every involution of $G$ lies in
$E$, the group $\mathrm{N}_G(Y)$ contains no involutions and hence has
odd order. Its quotient $\mathrm{N}_{H}(X)$ therefore also has odd order.
\end{proof}

\section{The abelian-socle case}

We now treat the abelian-socle case of Lemma~\ref{lem:monolithic}.
The key step is to prove that a Sylow $2$-subgroup centralizes the socle.
We establish this using Lemma~\ref{lem:normalizer}, Hering's
classification of transitive linear groups \cite{Hering}
(see also \cite[Appendix~1]{Liebeck}), and Alperin's theorem.
Lemma~\ref{lem:fusion-gap} then allows us to apply the theorem
of Li and Zhu.

We first record a module restriction needed in the $G_2$ case
of Lemma~\ref{lem:linear}. Using the split Cayley algebra,
we realize the restriction of the natural $G_2(q)$-module
to $\mathrm{SL}_3(q)$ as the direct sum of its natural module
and its dual.

Let $q$ be a power of $2$, and set $\mathbb{F}=\mathbb{F}_q$.
We use the Zorn matrix model of the split Cayley algebra
$\mathbf{O}=\mathbf{O}(\mathbb{F})$; see
\cite[Definition~2.1.9 and Example~2.1.12]{AHW}.
The underlying $8$-dimensional vector space consists of matrices
$a=\left(\begin{smallmatrix}\alpha&\mathbf{u}\\\mathbf{v}&\beta\end{smallmatrix}\right)$,
where $\alpha,\beta\in\mathbb{F}$ and
$\mathbf{u},\mathbf{v}\in\mathbb{F}^3$ are column vectors.
Multiplication is defined by
\[
\begin{pmatrix} \alpha & \mathbf{u} \\ \mathbf{v} & \beta \end{pmatrix}
\begin{pmatrix} \alpha' & \mathbf{u}' \\ \mathbf{v}' & \beta' \end{pmatrix}
=
\begin{pmatrix}
\alpha\alpha' + \mathbf{u}\cdot\mathbf{v}'
& \alpha\mathbf{u}' + \beta'\mathbf{u} - \mathbf{v}\times\mathbf{v}' \\
\alpha'\mathbf{v} + \beta\mathbf{v}' + \mathbf{u}\times\mathbf{u}'
& \beta\beta' + \mathbf{v}\cdot\mathbf{u}'
\end{pmatrix},
\]
where $\cdot$ and $\times$ denote the standard dot and cross products on
$\mathbb{F}^3$. The identity element is
$1_{\mathbf{O}}=\left(\begin{smallmatrix}1& \mathbf{0}\\ \mathbf{0}&1\end{smallmatrix}\right)$.
By \cite[Theorem~2.2.1]{AHW}, we may identify
$G_2(q)$ with $\mathrm{Aut}_{\mathbb{F}}(\mathbf{O})$.

Define the trace of $a$ by $\mathrm{tr}(a)=\alpha+\beta$.
Automorphisms of $\mathbf{O}$ preserve the trace and fix $1_{\mathbf{O}}$.
Thus $\mathbf{O}_0=\{a\in\mathbf{O}:\mathrm{tr}(a)=0\}$ is
$G_2(q)$-invariant and, since $\mathrm{char}(\mathbb{F})=2$,
contains the one-dimensional submodule $\mathbb{F}1_{\mathbf{O}}$.
Since $\dim\mathbf{O}_0=7$, the quotient
$V=\mathbf{O}_0/\mathbb{F}1_{\mathbf{O}}$ is the natural
$6$-dimensional $G_2(q)$-module over $\mathbb{F}$.
Every element of $\mathbf{O}_0$ has equal diagonal entries, so each coset
modulo $\mathbb{F}1_{\mathbf{O}}$ has a unique representative with both
diagonal entries zero. We therefore identify
$V$ with $\mathbb{F}^3\oplus\mathbb{F}^3$ by
\[
\begin{pmatrix} 0 & \mathbf{u} \\ \mathbf{v} & 0 \end{pmatrix} + \mathbb{F}1_{\mathbf{O}} \longleftrightarrow (\mathbf{u}, \mathbf{v}).
\]
For $A\in\mathrm{SL}_3(q)$, set
\[
\iota(A) \colon \begin{pmatrix} \alpha & \mathbf{u} \\ \mathbf{v} & \beta \end{pmatrix}
\longmapsto
\begin{pmatrix} \alpha & A\mathbf{u} \\ A^{-T}\mathbf{v} & \beta \end{pmatrix},
\]
where $A^{-T}:=(A^{-1})^T$. By \cite[Theorem~2.2.6]{AHW}, the assignment
$A\mapsto\iota(A)$ defines an embedding
$\iota:\mathrm{SL}_3(q)\hookrightarrow G_2(q)$.
Under the above identification, $\iota(A)$ sends
$(\mathbf{u},\mathbf{v})$ to $(A\mathbf{u},A^{-T}\mathbf{v})$.
Hence, if $W$ is the natural $3$-dimensional $\mathrm{SL}_3(q)$-module
and $W^*$ its dual, then $V=W\oplus W^*$ as a faithful
$\mathrm{SL}_3(q)$-module.

Recall that $G^{(\infty)}$ denotes the last term of the derived series of $G$.

\begin{theorem}\label{thm:hering}
Suppose that $G\leq \mathrm{GL}_n(2)$ acts transitively on the nonzero vectors of $V=\mathbb{F}_2^n$. Then one of the following holds:
\begin{enumerate}[label=\textup{(\roman*)}]
\item $G\leq \Gamma\mathrm{L}_1(2^n)$ is solvable;
\item 
$G^{(\infty)}$ is transitive on the nonzero vectors of $V$, and
\[
(G^{(\infty)},2^{n})\in \{ (\mathrm{SL}_{m}(q),q^{m}), (\mathrm{Sp}_{2m}(q),q^{2m}),(G_2(q),q^{6}) \}.
\]
Moreover, $V$ can be viewed as the natural $\mathbb{F}_{q}G^{(\infty)}$-module;
\item
$$(G,n)\in
\left\{
(\mathsf{A}_6,4),\,
(\mathrm{Sp}_4(2),4),\,
(\mathsf{A}_7,4),\,
(\mathrm{PSU}_3(3),6),\,
(G_2(2),6)
\right\}.$$
In these cases, the action of $G$ on $V$ is obtained from the corresponding natural low-dimensional modules as follows. For $G=\mathrm{Sp}_4(2)$, the module $V$ is its natural $4$-dimensional $\mathbb{F}_2$-module; for $G=\mathsf{A}_6\cong \mathrm{Sp}_4(2)'$, it is the restriction of this module to $\mathsf{A}_6$. 
For $G=\mathsf{A}_7$, it is the restriction of the natural $4$-dimensional $\mathbb{F}_2$-module of $\mathrm{SL}_4(2)\cong \mathsf{A}_8$. 
For $G=G_2(2)$, the module $V$ is its natural $6$-dimensional $\mathbb{F}_2$-module, while for $G=\mathrm{PSU}_3(3)\cong G_2(2)'$, it is the restriction of this $6$-dimensional module.
\end{enumerate}
\end{theorem}
\begin{proof}
  This is the part of Hering's classification needed here; see
\cite{Hering} or \cite[Appendix~1]{Liebeck}.
\end{proof}

\begin{lemma}\label{lem:linear}
Let $V=\F_2^{ n}$ be such that $n\ge2$. 
Suppose that $H\le\GL(V)$ is transitive on the nonzero vectors of $V$. 
If
\begin{equation}\label{eq:restriction}
 |\Norm{H}{X}|\text{ is odd whenever }X\le H\text{ has odd order and }\Cent{V}{X}=0,
\end{equation}
then either $|H|$ is odd, or $n=3$ and $H=\GL_3(2)$.
\end{lemma}
\begin{proof}
 By Theorem~\ref{thm:hering}, either
$H\leq\Gamma\mathrm{L}_1(2^n)$, or $H$ is non-solvable, in which case 
either 
$$(H^{(\infty)},2^n)
\in
\{
(\mathrm{SL}_m(q),q^m),
(\mathrm{Sp}_{2m}(q),q^{2m}),
(G_2(q),q^6)\},$$
or
$$(H,n)\in \{ (\mathsf{A}_6,4),
(\Sp_4(2),4),
(\mathsf{A}_7,4),
(\PSU_3(3),6),
(G_2(2),6) \}.$$
We consider these cases in turn. 
To exclude a case, it suffices to find a subgroup $L\leq H$ and an odd-order
subgroup $X\leq L$,
such that $\Cent{V}{X}=0$ and an involution in
$\Norm{L}{X}$.

\medskip\noindent\emph{The semilinear case.}
Suppose that $H\le\Gamma \mathrm{L}_1(2^n)$. 
Set $X=H\cap \mathrm{GL}_{1}(2^{n})$. If $X=1$, then $|H|\le n$, whereas transitivity gives $2^n-1\mid |H|$, which is impossible for $n\geq2$.
Hence $X\ne 1$. 
Every nonidentity scalar multiplication is fixed-point-free on $V^\#$,
so $\Cent{V}{X}=0$. 
Since $X\unlhd H$, condition~\eqref{eq:restriction} implies that $|H|$ is odd.

\medskip\noindent\emph{The special linear case.}
Suppose that $H^{(\infty)}=\mathrm{SL}_m(q)$.
By Theorem \ref{thm:hering}, we may regard $V$ as the natural $\mathbb{F}_qH^{(\infty)}$-module.
Let $\mathbb{X}$ be the subgroup of $\mathbb{F}_{q^{2}}^{\times}$ of order $q+1$,
and set $t=\lfloor m/2\rfloor$ and $s=m-2t$.
Then $s\in \{ 0,1 \}$.
We identify the natural module $V$ with $\mathbb{F}_4 \oplus \mathbb{F}_{2^{m-2}}$ when $t>s=1$ and $q = 2$,
and with 
$\mathbb{F}_{q^2}^{t-1} \oplus \mathbb{F}_{q^2} \oplus \mathbb{F}_q^s$ whenever either $s=0$ or $q > 2$.

Suppose that $s=1$ and $q=2$. 
If
$H^{(\infty)}=\mathrm{SL}_3(2)$, then $(H,n)=(\mathrm{GL}_3(2),3)$ because 
$\mathrm{SL}_3(2)=\mathrm{GL}_3(2)$.
For $t>1$, we have $V=\mathbb{F}_4 \oplus \mathbb{F}_{2^{m-2}}$ with $m>3$ odd.
Let 
\[
X=\{d_{\lambda,\mu}:\lambda\in\mathbb{F}_4^\times,
\ \mu\in\mathbb{F}_{2^{m-2}}^\times\}~\text{with}~(x,y)^{d_{\lambda,\mu}}=(\lambda x,\mu y).
\]
Then $X$ has odd order, and $\Cent{V}{X}=0$. 
Note that $\det(d_{\lambda,\mu})=\lambda^{3}\mu^{2^{m-2}-1}=1$,
and so $X\leq\mathrm{SL}_m(2)$.
The involution $\sigma$ given by
$(x,y)^\sigma=(x^2,y)$ normalizes $X$, again contradicting
condition~\eqref{eq:restriction}.

Now suppose that $s=0$.
Then $V=(\mathbb{F}_{q^{2}})^{ t}$.
Define
\begin{equation}\label{eq:x_2}
  X_{2t}=\{d_\lambda:\lambda\in \mathbb{X}\}~\text{with}~(x_1,\dots,x_{t})^{d_\lambda}=(\lambda x_1,\dots, \lambda x_{t}),
\end{equation}
for $(x_1,\dots, x_{t})\in V$.
For each $\lambda\in \mathbb{X}$,
since $\lambda^{q+1}=1$, $d_\lambda$ has determinant $1$. 
Hence $X_{2t}\leq H^{(\infty)}$ and $\mathrm{C}_{V}(X_{2t})=0$.
The map
\[
(x_1,\dots,x_t)^\sigma=(x_1^q,\dots,x_t^q)
\]
defines an involution $\sigma\in H^{(\infty)}$, and
$(d_\lambda)^\sigma=d_{\lambda^q}$. Thus
$\sigma\in\Norm{H^{(\infty)}}{X_{2t}}$, contrary to
condition~\eqref{eq:restriction}.

It remains to consider $s=1$ and $q>2$. 
Then $V=(\F_{q^2})^{ t-1}\oplus \mathbb{F}_{q^{2}}\oplus \mathbb{F}_q$.
Define 
\begin{equation}\label{eq:x_3}
X_{2t+1}=\{d_{\lambda,\mu}:\lambda\in \mathbb{X}, \mu\in\F_{q^2}^{\times}\}~\text{with}~(x_1,\dots,x_{t-1},x_{t},z)^{d_{\lambda,\mu}}=(\lambda x_1,\dots, \lambda x_{t-1},\mu x_{t}, \mu^{-(q+1)}z),
\end{equation}
for $(x_1,\dots, x_{t},z)\in V$.
It is straightforward to check that $X_{2t+1}\leq H^{(\infty)}$ and $\mathrm{C}_{V}(X_{2t+1})=0$.
Define $\sigma\in H^{(\infty)}$ by 
$$(x_1,\dots, x_{t},z)^{\sigma}=(x_1^{q},\dots, x_{t}^{q},z^{q}).$$
Then $\sigma$ is an involution and $(d_{\lambda,\mu})^{\sigma}=d_{\lambda^q,\mu^{q}}$.
Hence $\sigma\in\Norm{H^{(\infty)}}{X_{2t+1}}$, again contradicting
condition~\eqref{eq:restriction}.

\medskip\noindent\emph{The symplectic case.}
Suppose that $H^{(\infty)}=\Sp_{2m}(q)$. 
By Theorem \ref{thm:hering}, we regard $V$ as the natural module of $H^{(\infty)}$ over $\mathbb{F}_q$.
Decompose $V$ as an orthogonal sum
of $m$ hyperbolic planes, each identified with $\mathbb{F}_{q^2}$. The
subgroup fixing every summand contains $\mathrm{Sp}_2(q)^m$, acting
naturally on the respective planes.
Since $\mathrm{Sp}_2(q)=\mathrm{SL}_2(q)$,
we take $X=(X_2)^{m}$, where $X_2$ is defined in \eqref{eq:x_2} with $t=1$.
Then $\mathrm{C}_{V}(X)=0$, 
\[
 \mathrm{N}_{\mathrm{SL}_2(q)}(X_2)^{m}= \mathrm{N}_{\mathrm{Sp}_2(q)^{m}}(X)\leq \mathrm{N}_{H}(X).
\]
The special linear case shows that $2\mid |\mathrm{N}_{\mathrm{SL}_2(q)}(X_2)|$,
and hence $2 \mid |\mathrm{N}_{H}(X)|$.
This excludes the symplectic case, including
$(H,n)=(\mathrm{Sp}_4(2),4)$.

\medskip\noindent\emph{The $G_2$ case.} Suppose that $H^{(\infty)}=G_2(q)$.
Then $q>2$.
By Theorem \ref{thm:hering}, we may view $V$ as the natural $H^{(\infty)}$-module over $\mathbb{F}_q$.
The preceding discussion provides a subgroup $L\leq H^{(\infty)}$ isomorphic to $\mathrm{SL}_3(q)$
such that, as an $L$-module, 
\[
V=W\oplus W^*,
\] 
where $W$ is the natural $L$-module and $W^*$ is its dual.
Identifying $L$ with $\mathrm{SL}_3(q)$,
we set $X=X_3$, where $X_3$ is defined in
\eqref{eq:x_3} with $t=1$. 
Then $\mathrm{C}_{V}(X)=\mathrm{C}_{W}(X)\oplus\mathrm{C}_{W^*}(X)=0$.
The special linear case gives $2\mid |\mathrm{N}_{\mathrm{SL}_3(q)}(X_3)|$,
and therefore $2\mid |\mathrm{N}_{H^{(\infty)}}(X)|$, 
which contradicts~\eqref{eq:restriction}.

\medskip\noindent\emph{The four-dimensional exceptions.}
Suppose that $(H,n)\in \{ (\mathsf{A}_6,4),(\mathsf{A}_7,4) \}$.
Let $X$ be a subgroup of $H$ of order $5$.
Since the multiplicative order of $2$ modulo $5$ is $4$, 
the restriction to $X$ of a faithful
$4$-dimensional representation over $\F_2$ has no trivial constituent.
Thus $\mathrm{C}_{V}(X)=0$.
A $5$-cycle is inverted by a double transposition
in the corresponding alternating group, so $\Norm{H}{X}$ has even order.
This contradicts condition~\eqref{eq:restriction}.

\medskip\noindent\emph{The six-dimensional exceptions.}
It remains to consider $(H,n)\in \{ (\mathrm{PSU}_{3}(3),6),(G_2(2),6) \}$.
Since $G_2(2)'\cong \mathrm{PSU}_{3}(3)$ is non-abelian simple,
we have $(H^{(\infty)},n)=(\mathrm{PSU}_{3}(3),6)$.
By the $\mathbb{ATLAS}$ \cite{Atlas}, the group $H^{(\infty)}$ has a maximal subgroup $M\cong \mathsf{ES}(3_+^{3}) \rtimes \mathsf{C}_{8}$.
Let $X$ be a Sylow $3$-subgroup of $M$.
Then $\mathrm{N}_{H^{(\infty)}}(X)$ has even order, and $V$ is a faithful $\mathbb{F}_2 X$-module.
Since $\mathbb{F}_2$ is not a splitting field for $X$ but $\mathbb{F}_4$ is, all faithful irreducible $\mathbb{F}_2X$-modules have dimension $6$.
 Hence $V$ is a faithful irreducible $\mathbb{F}_2X$-module, 
 and therefore $\mathrm{C}_{V}(X)=0$.
This is the final contradiction.
\end{proof}

The only even-order group allowed by Lemma~\ref{lem:linear}
is $\GL_3(2)$. The following consequence of Alperin's theorem
excludes its occurrence as the induced action on a normal
elementary abelian subgroup containing all involutions.

\begin{lemma}\label{lem:exclude}
There is no group $G$ with $\OddCore{G}=1$ and a normal subgroup $E\cong (\mathsf{C}_2)^3$ such that $\Aut_G(E)\cong\GL_3(2)$ and every involution of $G$ belongs to $E$. 
\end{lemma}
\begin{proof}
Suppose that such a group $G$ exists, and set $C=\Cent{G}{E}$.
Then $E\le C$, and the involutions of $C$ are exactly the
nonidentity elements of $E$. Moreover,
\[
G/C\cong\Aut_G(E)\cong\GL_3(2)=\Aut(E).
\]
Since $\OddCore{G}=1$, all the hypotheses of
\cite[Theorem 1]{Alperin} are satisfied.
By \cite[Corollary 2]{Alperin}, there is an involution in
$G\setminus C$. This contradicts the assumption that every
involution of $G$ belongs to $E$.
\end{proof}

We now apply Lemmas~\ref{lem:linear} and \ref{lem:exclude}
to complete the abelian-socle case.

Recall that a finite $2$-group $P$ is a Suzuki $2$-group if it is
nonabelian, has more than one involution, and admits an automorphism
of odd order acting regularly on its involutions.

\begin{theorem}\label{thm:solvable}
Let $G\in \mathbf{OC}_2$. 
If $E:=\mathrm{soc}(G)$ is solvable and $\mathrm{O}_{2'}(G)=1$,
then $P\in \mathrm{Syl}_{2}(G)$ is homocyclic, generalized quaternion, or a Suzuki $2$-group. 
Conversely, every nontrivial group in
one of these three families occurs as a Sylow $2$-subgroup of a finite
group $G$ satisfying these hypotheses.
\end{theorem}

  \begin{proof}
By Lemma~\ref{lem:monolithic}, $E = \mathrm{soc}(G) = \mathcal{I}(G) \cup \{1\} \cong \mathsf{C}_2^n$. If $n = 1$, then every Sylow $2$-subgroup of $G$ has a unique involution and is therefore cyclic or generalized quaternion. 
 We may thus assume that $n\geq2$.

Let $P \in \mathrm{Syl}_2(G)$, and set $C = \mathrm{C}_G(E)$, $S = P \cap C$, $N = \mathrm{N}_G(S)$.
Since $C \unlhd G$, we have $S \in \mathrm{Syl}_2(C)$, $S \unlhd P$, and $P \leq N$. 
The Frattini argument then gives $G = CN$.
 Moreover, $E \leq S$, and $C$ acts trivially on $E$. 
 Hence $N$ acts transitively on $E^\#$. Since every involution of $G$ lies in $E^\#$, it follows that $\mathcal{I}(N)=\mathcal{I}(S) = E^\#$ and $\Omega_1(S) = E$.

Set $D=N\cap C=\Cent{N}{E}$ and $\mathcal{H}=N/D$. Then
$\mathcal{H}\leq\GL(E)$ acts transitively on $E^\#$. We verify
condition~\eqref{eq:restriction} of Lemma~\ref{lem:linear}. Let
$X\leq\mathcal{H}$ have odd order and satisfy $\Cent{E}{X}=1$. The group
$D$ has the normal Sylow $2$-subgroup $S$, with $\Omega_1(S)=E$;
moreover, $D=\Cent{N}{E}$ and $\mathcal{I}(N)\subseteq E$.
Lemma~\ref{lem:normalizer}, applied to the pair $(N,D)$, therefore shows
that $\Norm{\mathcal{H}}{X}$ has odd order.

Lemma~\ref{lem:linear} now implies that either $\mathcal{H}$ has odd order, or $n = 3$ and $\mathcal{H} \cong \mathrm{GL}_3(2)$. 
In the latter case,
\[
\Aut_G(E)=G/C\cong N/D=\mathcal{H}\cong\GL_3(2),
\]
contrary to Lemma~\ref{lem:exclude}.
Thus $|G : C|=|N:D|=|\mathcal{H}|$ is odd. 
Since $$P/S = P/(P \cap C) \cong PC/C \leq G/C$$ 
is a $2$-group, we have $P = S$.
Consequently, $N = \mathrm{N}_G(P)$ acts transitively on $\mathcal{I}(P) = E^\#$. 
By Lemma \ref{lem:fusion-gap},
$\mathrm{Aut}(P)$ is transitive on the involutions of $P$.
It follows from
\cite[Theorem~1.1]{LZ} that $P$ is homocyclic, generalized quaternion,
or a Suzuki $2$-group.
Since $n \geq 2$, the group $P$ has at least three involutions and hence is not generalized quaternion. 
Together with the case $n=1$, this proves the forward implication.

\medskip

For the converse, if $P$ is generalized quaternion, we simply take $G=P$. 
Now
suppose that $P$ is homocyclic or a Suzuki $2$-group, and set
$E=\Omega_1(P)$. Then
$E=\mathcal{I}(P)\cup\{1\}$ is an elementary abelian 2-group. 
In the Suzuki case,
this follows because $\mathrm{Z}(P)$ contains an involution and is
characteristic in $P$, so transitivity forces every involution of $P$ to
be central.

We claim that $\Aut(P)$ contains an odd-order cyclic subgroup $B$ acting
transitively on $E^\#$. Suppose first that
$P\cong(\mathsf{C}_{2^a})^n$. 
Restriction of automorphisms yields an epimorphism
\[
\Aut(P)\longrightarrow \mathrm{Aut}(\Omega_1(P))
\]
whose kernel is a $2$-group. The preimage of a cyclic subgroup of
$ \mathrm{Aut}(\Omega_1(P))\cong\GL_n(2)$ of order $2^n-1$ therefore has, by the Schur--Zassenhaus
theorem, a cyclic complement $B$ of order $2^n-1$. Its induced action on
$E^\#$ is regular. If $P$ is a Suzuki $2$-group, take the odd-order
cyclic subgroup $B\leq\Aut(P)$ provided by the definition.

Set $G=P\rtimes B$. Since $G/P$ has odd order, every involution of $G$
lies in $P$. Thus $\mathcal{I}(G)=E^\#$, which is a single
$G$-conjugacy class.

It remains to identify the socle and the odd core. Let
$1\neq M\unlhd G$. If $M\cap P\neq1$, then the $2$-group $M\cap P$
contains an involution. Hence $M\cap E\neq1$, and the transitivity of $B$
on $E^\#$ gives $E\leq M$. If $M\cap P=1$, then
$[M,P]\leq M\cap P=1$, so $M\leq\Cent{G}{P}$. Since $B$ has odd order
and $\Inn(P)$ is a $2$-group, no nonidentity element of $B$ induces an
inner automorphism of $P$. Therefore $\Cent{G}{P}=\mathrm{Z}(P)\leq P$,
contrary to $M\cap P=1$. Thus every nontrivial normal subgroup of $G$
contains $E$. It follows that $E$ is the unique minimal normal subgroup,
$\soc(G)=E$, and $\mathrm{O}_{2'}(G)=1$. This completes the proof.
\end{proof}

\section{The almost simple case}

We now consider almost simple groups $A$ with socle $S$. We first
determine which simple groups have a single involution class and show
that an extension of odd index belongs to $\mathbf{OC}_2$ precisely
when its socle does. We then treat extensions of even index and
determine the Sylow $2$-subgroups that arise.

For a simple group $S$ of Lie type, $\Inndiag(S)$ denotes the
subgroup generated by the inner and diagonal automorphisms.
By \cite[Theorem~2.5.12]{GLS},
\begin{equation}\label{eq:aut_1}
   \mathrm{Aut}(S)=\Inndiag(S) \rtimes \Phi \Gamma,
\end{equation}
where $\Phi$ induces field automorphisms and $\Gamma$ induces 
graph automorphisms,
with the conventions of \cite[\S 2.5]{GLS}.

In the notation $\mathrm{GL}_n^{\epsilon}(q)$,
$\mathrm{SL}_n^{\epsilon}(q)$, and $\mathrm{PSL}_n^{\epsilon}(q)$,
the sign $\epsilon=+$ denotes the corresponding linear group, and
$\epsilon=-$ denotes $\mathrm{GU}_n(q)$, $\mathrm{SU}_n(q)$, and
$\mathrm{PSU}_n(q)$, respectively.

\begin{theorem}[Yamaki]\label{thm:yamaki}
If a nonabelian finite simple group $T$ has at least two conjugacy classes of involutions, then its involutions do not form a single orbit under $\Aut(T)$.
\end{theorem}
\begin{proof}
  This is \cite[Lemma~2]{Yamaki}, see also \cite[Lemma~3.4]{CJ}.
\end{proof}

Applying Theorem \ref{thm:yamaki}, we have the following lemma.

\begin{lemma}
\label{lem:socle-inheritance}
Let $A$ be an almost simple group with socle $S$.
If $A\in \mathbf{OC}_2$, then 
$S\in \mathbf{OC}_2$.
\end{lemma}
\begin{proof}
By Lemma \ref{lem:monolithic}(ii), $\mathcal{I}(A)=\mathcal{I}(S)$. Since $A\in\mathbf{OC}_2$, the involutions of
$S$ form one orbit under $A$, and hence under $\Aut(S)$. Theorem 4.1 implies
that $S\in\mathbf{OC}_2$. 
\end{proof}

\begin{lemma}\label{lem:simple-list}
Let $A$ be an almost simple group with socle $S$ and suppose that
$|A:S|$ is odd. Then the following are equivalent:
\begin{enumerate}[label=\textup{(\roman*)}]
\item $A\in\mathbf{OC}_2$;
\item $S\in\mathbf{OC}_2$;
\item $S$ occurs in Table~\ref{tab:simple-list}.
\end{enumerate}
\end{lemma}

\begin{table}[ht]
\centering
\caption{Nonabelian simple groups with one involution class}
\label{tab:simple-list}
\renewcommand{\arraystretch}{1.17}
\begin{tabularx}{\textwidth}{@{}>{\bfseries}l X@{}}
\toprule
Family & Simple groups $S$ \\
\midrule
Classical
& $\mathrm{PSL}_2(q)$, $q\geq4$;
  $\mathrm{PSL}_3(q)$, $q\geq2$;
  $\PSU_3(q)$, $q\geq3$;
  $\mathrm{PSL}_4(q)$, $q\equiv5\pmod 8$;
  $\PSU_4(q)$, $q\equiv3\pmod 8$. \\
Exceptional
& ${}^{2}B_2(2^{2a+1})$, $a\geq1$;
  ${}^{2}G_2(3^{2a+1})$, $a\geq1$;
  $G_2(q)$, $q$ odd;
  ${}^{3}D_4(q)$, $q$ odd. \\
Alternating
& $\mathsf{A}_n$, $5\leq n\leq7$. \\
Sporadic
& $M_{11},M_{22},M_{23},J_1,J_3,McL,O'N,
   Ly,Th$. \\
\bottomrule
\end{tabularx}
\end{table}
\begin{proof}
Since $|A : S|$ is odd, we have $\mathcal{I}(A) = \mathcal{I}(S)$. 
Thus (ii) implies (i).
Applying Lemma \ref{lem:socle-inheritance}, 
we have that (i) implies (ii).
Finally, an application of \cite[Proposition~3.1]{CJ} shows that (ii) and (iii) are equivalent.
\end{proof}

\begin{remark}\label{rem:accidental}
The table lists familiar family names rather than disjoint isomorphism
types.  In particular
\[
 \mathsf{A}_5\cong\mathrm{PSL}_2(4)\cong\mathrm{PSL}_2(5),\qquad
 \mathsf{A}_6\cong\mathrm{PSL}_2(9),\qquad
 \mathrm{PSL}_3(2)\cong\mathrm{PSL}_2(7).
\]
Also ${}^{2}G_2(3)'\cong\mathrm{PSL}_2(8)$ and
$G_2(2)'\cong\PSU_3(3)$.  Thus a nonredundant alternating contribution
adds only $\mathsf{A}_7$.
\end{remark}

Lemma~\ref{lem:simple-list} settles the odd-index case.
Suppose now that $A\in\mathbf{OC}_2$ and $|A:S|$ is even.
By Lemmas~\ref{lem:monolithic} and \ref{lem:socle-inheritance},
every involution of $A$ lies in $S$, and $S$ occurs in
Table~\ref{tab:simple-list}. We use these restrictions to reduce
the even-index case to type $A_1$ in odd characteristic.

\begin{proposition}
\label{prop:even-index-reduction}
  Let $A$ be an almost simple group with socle $S$.
  Suppose that $A\in \mathbf{OC}_2$.
  If $|A:S|$ is even,
  then $S\cong \mathrm{PSL}_{2}(p^{f})$ for an odd prime $p$ and an even integer
$f$.
\end{proposition}

\begin{proof}
By Lemmas~\ref{lem:socle-inheritance} and \ref{lem:simple-list}, the group $S$ occurs in
Table~\ref{tab:simple-list}. Suppose first that
$S\cong\mathsf{A}_n$ with $n\geq7$. Since $|A:S|$ is even, we have
$A\cong\mathsf{S}_n$, which contains involutions outside
$\mathsf{A}_n$. This contradicts Lemma~\ref{lem:monolithic}(ii), so no
such alternating group occurs.

For each sporadic group $S$ in Table~\ref{tab:simple-list}, the outer
automorphism group has order at most $2$. When it has order $2$, 
the involution of $\mathrm{Out}(S)$ lifts to an involution
of $\mathrm{Aut}(S)$, as recorded in the $\mathbb{ATLAS}$
\cite{Atlas}. Lemma~\ref{lem:monolithic}(ii) therefore excludes every
sporadic group in the table.

Since $\mathsf{A}_5\cong\mathrm{PSL}_2(5)$ and
$\mathsf{A}_6\cong\mathrm{PSL}_2(9)$, it remains to consider the groups
of Lie type in Table~\ref{tab:simple-list}. By
\cite[Theorem~2.5.12]{GLS},
\begin{equation}\label{eq:aut}
  \Aut(S)=G\rtimes\langle\phi\rangle\Gamma,
\end{equation}
where $G:=\Inndiag(S)$, $\phi$ is a field automorphism and $\Gamma$ is generated by graph
automorphisms.

First suppose that $S=\mathrm{PSL}_4^{\epsilon}( q)$, where
$\epsilon\in\{\pm\}$ and $4\parallel q-\epsilon1$. By
\cite[Theorems~1.12 and~2.9]{LMM}, the group $\Aut(S)$ splits over $S$,
and hence so does $A$. Since $2\mid|A:S|$, the group $A$ contains an
involution outside $S$, contrary to Lemma~\ref{lem:monolithic}(ii).

Now suppose that $|G:S|$ is odd. 
Choose an involution
$xS\in A/S$. 
Note that 
$$\operatorname{gcd}(|\langle x\rangle G/G|,|G/S|)=\operatorname{gcd}(2,|G/S|)=1,$$
and so an application of the Schur-Zassenhaus theorem to $\langle x\rangle G/S$ forces $\langle x\rangle G/S=G/S  \rtimes \langle xS\rangle$.
As $x\notin G$, by \eqref{eq:aut}, 
$\langle x\rangle G$ splits over $G$,
so that there exists an involution $y\in \langle x\rangle G$ 
such that $y^{g}\in y^gS=xS \subseteq A$ for some $g\in G$.
Thus $A$ again contains an involution
outside $S$, contradicting Lemma~\ref{lem:monolithic}(ii). Consequently,
$|G:S|$ is even. 
Inspection of the remaining groups of Table \ref{tab:simple-list} in the
$\mathbb{ATLAS}$ \cite{Atlas} yields
\[
S\cong\mathrm{PSL}_2(p^f)
\]
with $p$ odd and $p^f>4$. Finally, \cite[Theorem~1.12]{LMM} shows that
$\Aut(S)$ splits over $S$ when $f$ is odd. The same argument would then
produce an involution in $A\setminus S$. Hence $f$ is even.
\end{proof}

Proposition~\ref{prop:even-index-reduction} leaves the case
$\mathrm{PSL}_2(p^{2a})$ with $p$ odd.
Before classifying the possible extensions, we define the mixed
diagonal--field extension and determine its Sylow $2$-subgroups.
Let $S=\mathrm{PSL}_2(q)$, where
$q=q_0^2$ and $q_0=p^a$ for an odd prime $p$. Let $\phi$ be the field
automorphism induced by $t\mapsto t^p$ on $\mathbb{F}_q$, and set
$\sigma=\phi^a$. 
Thus $\sigma$ is the involutory field automorphism
induced by $t\mapsto t^{q_0}$.

Let $\pi:\Aut(S)\to\Out(S)$ be the natural quotient map, and identify
$\mathrm{PGL}_2(q)$ with $\Inndiag(S)$. Choose a nonsquare
$\eta\in\mathbb{F}_q^\times$, set $d=\mathrm{diag}(\eta,1)$, and let
$\delta$ be the diagonal automorphism induced by $d$.
Writing $Z=\mathrm{Z}(\mathrm{SL}_2(q))$, we have
$(hZ)^\delta=d^{-1}hd\,Z$ for every $h\in\mathrm{SL}_2(q)$. 
Write $\overline\delta=\pi(\delta)$ and
$\overline\sigma=\pi(\sigma)$. Following \cite[p.~42]{KLS}, we define
\[
\mathrm{PGL}_2^*(p^{2a})
=\pi^{-1}\bigl(\langle\overline{\delta\sigma}\rangle\bigr).
\]

We next determine its Sylow $2$-subgroups. Set $2^b=(q-1)_2$. Since
$q=q_0^2$ with $q_0$ odd, $q\equiv1\pmod8$ and hence
$(q+1)_2=2$. Therefore
\[
|\mathrm{PSL}_2(q)|_2
=\left(\frac{q(q-1)(q+1)}{2}\right)_2
=(q-1)_2=2^b,
\]
so $b\geq3$. Set $Y=\mathrm{PGL}_2^*(p^{2a})$. 
Since $\overline{\delta\sigma}$ has order $2$ in $\mathrm{Out}(S)$, we have $|Y:S| = 2$ and $|Y|_2 = 2^{b+1} = 2(q-1)_2$.

For $A\in\mathrm{GL}_2(q)$, let $\widehat A$ denote its image in
$\mathrm{PGL}_2(q)$. Choose $\beta\in\mathbb{F}_q^\times$ of order
$2^b$, and set
\[
y=\widehat{\mathrm{diag}(\beta,1)},\qquad
w=\widehat{\begin{pmatrix}0&1\\-1&0\end{pmatrix}}\in S.
\]
Since $o(\beta)=2^b=(q-1)_2$, the element $\beta$ is a nonsquare in
$\mathbb{F}_q^\times$. Thus $y$ induces the nontrivial diagonal outer
automorphism and $\pi(y)=\overline\delta$. Moreover,
\[
o(y)=2^b,\qquad y^\sigma=y^{q_0},\qquad wyw=y^{-1},
\qquad [w,\sigma]=1,\qquad w^2=1.
\]
Define
\[
x=
\begin{cases}
y\sigma, & \text{if } q_0 \equiv 1 \pmod{4}, \\
yw\sigma, & \text{if } q_0 \equiv -1 \pmod{4}.
\end{cases}
\]
In either case, $\pi(x)=\overline{\delta\sigma} $, so $x\in Y$.
A direct calculation using the preceding relations gives
$o(x) = 2^b$ and $wxw = x^{\,2^{b-1}-1}$.
Moreover, $w\notin\langle x\rangle$, and hence
$|\langle x,w\rangle|=2^{b+1}=|Y|_2$. Thus $\langle x,w\rangle$ is a
Sylow $2$-subgroup of $Y$, with presentation
\[
\langle x,w :
x^{2^b}=w^2=1,~wxw=x^{2^{b-1}-1}
\rangle
\cong\mathsf{SD}_{2^{b+1}}.
\]

\begin{theorem}\label{thm:A1-even}
Let $A$ be an almost simple group with socle $S$.
Suppose that $2\mid|A:S|$. Then $A\in\mathbf{OC}_2$ if and only if
there are an odd prime $p$, an
integer $a\geq 1$, and an odd divisor $d$ of $a$ such that
\[
 S\cong\mathrm{PSL}_2(p^{2a}),\qquad
 A\cong\PGL_2^*(p^{2a})\rtimes\langle\phi^{2a/d}\rangle,
\]
where $\phi$ is the field automorphism induced by $t\mapsto t^p$ on
$\mathbb{F}_{p^{2a}}$. In this case, every $P\in\Syl_2(A)$ satisfies
$P\cong\SD_{2(p^{2a}-1)_2}$.
\end{theorem}

\begin{proof}
Assume first that $A \in \mathbf{OC}_2$. By Lemma \ref{lem:monolithic}(ii)
and Proposition~\ref{prop:even-index-reduction}, we have $S \cong \mathrm{PSL}_2(p^{2a})$ for some odd prime $p$ such that $\mathcal{I}(S)=\mathcal{I}(A)$.
Write $\mathrm{Aut}(S) = G \rtimes \langle \phi \rangle$ where $G : = \mathrm{Inndiag}(S)$, and identify $G$ with $\mathrm{PGL}_2(p^{2a})$.
Then $|G:S|=2$.
Since $G$ splits over $S$, it forces $G\cap A = S$.
Let $\langle \sigma\rangle$ be the Sylow 2-subgroup of $GA\cap\langle \phi\rangle(\cong GA/G\cong A/S)$.
We claim that $o(\sigma)=2$.
In fact, otherwise $o(\sigma)=2^{k}>2$;
since $\sigma S=gS \cdot aS$ for some $g\in G$ and $a\in A$ with $o(gS)\leq 2$,
and $GA/S$ is an abelian group,
it follows that $\sigma^{2^{k-1}}S=a^{2^{k-1}}S\in A/S$, whence $\sigma^{2^{k-1}}$ is an involution in $A\setminus S$; this contradicts $\mathcal{I}(S)=\mathcal{I}(A)$.
As $|A:S|=|AG:G|$, 
we conclude that 
$$|A:S|_2=|AG:G|_2=o(\sigma)=2.$$

Choose $P\in\Syl_2(A)$, and set $Q=P\cap S$. Then
$Q\in\Syl_2(S)$. By \cite[Chapter~2, Theorem~8.3]{Gorenstein}, the
Sylow $2$-subgroups of $\mathrm{SL}_2(p^{2a})$ are generalized
quaternion. Since
$\mathrm{Z}(\mathrm{SL}_2(p^{2a}))\cong\mathsf{C}_2$, it follows that
$Q\cong\mathsf{D}_{2^b}$, where $2^b=|S|_2$.
Since every involution of $A$ lies in $S$, every element of
$P\setminus Q$ has order greater than $2$, so that
\[
|\mathcal{I}(P)|=|\mathcal{I}(Q)|=2^{b-1}+1\equiv1\pmod4.
\]
By \cite[Theorem~4.9]{Isaacs}, $P'$ has index $4$ in $P$. Therefore,
by \cite[Chapter~5, Theorem~4.5]{Gorenstein},
\[
P\in\{\mathsf{D}_{2^{b+1}},\mathsf{Q}_{2^{b+1}},
\mathsf{SD}_{2^{b+1}}\}.
\]
The group $P$ has the maximal subgroup
$Q\cong\mathsf{D}_{2^b}$ and no involution in $P\setminus Q$.
Consequently, $P\cong\mathsf{SD}_{2^{b+1}}$.

The group $PS$ is isomorphic to neither $\mathrm{PGL}_2(p^{2a})$ nor
$S\rtimes\langle\phi^a\rangle$, since each of those groups contains an
involution outside $S$. Hence
$PS=\mathrm{PGL}_2^*(p^{2a})$. Moreover,
\[
A/S=PS/S\times\langle\phi^{2a/d}S\rangle
\]
for some odd divisor $d$ of $a$. Since
$PS\cap\langle\phi^{2a/d}\rangle
\leq S\cap\langle\phi^{2a/d}\rangle=1$,
we obtain
$A=PS\rtimes\langle\phi^{2a/d}\rangle$.

\medskip
Conversely, suppose that $A$ and $S$ have the stated properties. The calculation preceding this theorem shows that every $P\in\Syl_2(A)$ satisfies
$P\cong\mathsf{SD}_{2(p^{2a}-1)_2}$.
Set $Q=P\cap S$. Then $Q\in\Syl_2(S)$ and
$Q\cong\mathsf{D}_{(p^{2a}-1)_2}$. 
Every $2$-element of $A\setminus S$
lies in $P^g\setminus Q^g$ for some $g\in A$. 
The semidihedral
structure of $P$ shows that no such element is an involution. Finally,
$S\in\mathbf{OC}_2$ by Lemma~\ref{lem:simple-list}, and hence
$A\in\mathbf{OC}_2$.
\end{proof}

Combining Lemma~\ref{lem:simple-list} with
Theorem~\ref{thm:A1-even} gives the classification for both parities of
the outer index.

\begin{theorem}\label{thm:almost-simple}
Let $A$ be an almost simple group with socle $S$.  
Then
$A\in\mathbf{OC}_2$ if and only if exactly one of the following holds:
\begin{enumerate}[label=\textup{(\roman*)}]
\item $|A:S|$ is odd and $S$ occurs in Table~\ref{tab:simple-list};
\item $S\cong \mathrm{PSL}_{2}(p^{2a})$ and
$A\cong\PGL_2^*(p^{2a})\rtimes \langle\phi^{2a/d}\rangle$,
where $p$ is an odd prime, $a\geq1$, $d$ is an odd divisor of $a$, and 
$\phi$ is the field automorphism of $\mathrm{PSL}_{2}(p^{2a})$ induced by $t\mapsto t^{p}$ on $\mathbb{F}_{p^{2a}}$.
\end{enumerate}
\end{theorem}

\begin{proof}
Suppose first that $A\in\mathbf{OC}_2$. If $|A:S|$ is odd, then
Lemmas~\ref{lem:socle-inheritance} and \ref{lem:simple-list}
give (i). If $|A:S|$ is even, Theorem~\ref{thm:A1-even} gives precisely
the groups in (ii).

Conversely, assume (i). Lemma~\ref{lem:simple-list} gives
$S\in\mathbf{OC}_2$, and the converse direction of
Lemma~\ref{lem:simple-list} yields $A\in\mathbf{OC}_2$. If (ii)
holds, the converse direction of Theorem~\ref{thm:A1-even} gives the
same conclusion.
\end{proof}

\begin{lemma}\label{lem:mixed-redundant}
  Let $q$ be a power of an odd prime $p$, and let $\epsilon\in \{ \pm \}$.
  If $q\equiv -\epsilon 1 \pmod 4$,
  then $\mathsf{SD}_{2(q^{2}-1)_2}$ is isomorphic to a Sylow $2$-subgroup of $\mathrm{PSL}^{\epsilon}_{3}(q)$.
\end{lemma}

\begin{proof}
Let $H=\mathrm{SL}_3^{\epsilon}(q)$ and
$Z=\mathrm{Z}(H)\cong\mathsf{C}_d$, where
$d=\gcd(3,q-\epsilon1)$. Choose $P\in\Syl_2(H)$. Since $|Z|$ is odd,
the image of $P$ in
$H/Z=\mathrm{PSL}_3^{\epsilon}(q)$ is a Sylow $2$-subgroup isomorphic to
$P$.

The map $x\mapsto\mathrm{diag}(x,\det(x)^{-1})$ embeds
$\mathrm{GL}_2^{\epsilon }(q)$ into $\mathrm{SL}_3^{\epsilon }(q)$. Moreover,
\[
|\mathrm{GL}_2^{\epsilon }(q)|_2
=(q^2-1)_2(q-\epsilon1)_2
=|\mathrm{SL}_3^{\epsilon}( q)|_2,
\]
so $P$ is isomorphic to a Sylow $2$-subgroup of
$\mathrm{GL}_2^{\epsilon}( q)$.

Since $q\equiv-\epsilon1\pmod4$, we have
$(q-\epsilon1)_2=2$ and therefore
$|\mathrm{GL}_2^{\epsilon}( q)|_2=2(q^2-1)_2$. 
Consequently, it follows from
\cite[pp.~142--143]{CF} that
$P\cong\mathsf{SD}_{2(q^2-1)_2}$.
\end{proof}

Lemma~\ref{lem:mixed-redundant} shows that the semidihedral
Sylow $2$-subgroups arising at even outer index already occur
in simple groups in Table~\ref{tab:simple-list}.
Thus the almost simple case contributes precisely the Sylow
$2$-subgroups described in the following corollary.

\begin{corollary}
\label{cor:almost-sylow}
A finite $2$-group $P$ occurs as a Sylow $2$-subgroup of an almost simple
group in $\mathbf{OC}_2$ if and only if $P$ is isomorphic to a Sylow
$2$-subgroup of a simple group in Table~\ref{tab:simple-list}.
\end{corollary}

\begin{proof}
Let $P\in\Syl_2(A)$, where $A\in\mathbf{OC}_2$ is almost simple with
socle $S$. If case (i) of Theorem~\ref{thm:almost-simple} holds, then
$|A:S|$ is odd, so $P\in\Syl_2(S)$ and $S$ occurs in
Table~\ref{tab:simple-list}. If case (ii) holds, then
Theorem~\ref{thm:A1-even} and Lemma~\ref{lem:mixed-redundant} show that
$P$ is isomorphic to a Sylow $2$-subgroup of either
$\mathrm{PSL}_3(p^a)$ or $\mathrm{PSU}_3(p^a)$, both of which occur in
Table~\ref{tab:simple-list}.

Conversely, if $P$ is isomorphic to a Sylow $2$-subgroup of a simple
group $S$ in Table~\ref{tab:simple-list}, then
Lemma~\ref{lem:simple-list} gives $S\in\mathbf{OC}_2$, proving
the result.
\end{proof}

\section{Proof of the main result}

\begin{proof}[Proof of Theorem~\ref{thm:brauer}]
Suppose first that $G\in\mathbf{OC}_2$ and $P\in\Syl_2(G)$. 
Set
$N=\OddCore{G}$ and $\overline G=G/N$. 
By Lemma~\ref{lem:oddcore}, we also have
$\overline G\in\mathbf{OC}_2$.

Lemma~\ref{lem:monolithic} now gives two cases. If
$\soc(\overline G)$ is elementary abelian, then
Theorem~\ref{thm:solvable} shows that $P$ is homocyclic, generalized
quaternion, or a Suzuki $2$-group. Otherwise, $\overline G$ is almost
simple, and Corollary~\ref{cor:almost-sylow} shows that $P$ is
isomorphic to a Sylow $2$-subgroup of one of the simple groups in
Table~\ref{tab:simple-list}. This proves the necessity of conditions
\textup{(i)}--\textup{(iv)}.

Conversely, the converse direction of Theorem~\ref{thm:solvable}
realizes every nontrivial group in \textup{(i)}--\textup{(iii)} as a
Sylow $2$-subgroup of a group in $\mathbf{OC}_2$. If \textup{(iv)}
holds, let $S$ be the corresponding simple group in
Table~\ref{tab:simple-list}. Lemma~\ref{lem:simple-list} gives
$S\in\mathbf{OC}_2$, so $S$ itself provides the required realization.
\end{proof}

\section*{Acknowledgements}
The first author was supported by  NSF of China (Grant No. 11871011) and is now supported by Huaqiao University Research Startup Funds (Grant No. 605-50Y21007);
the second author is supported by NSF of China (Grant No. 12671032).

The authors are also grateful to the referee for her/his valuable comments.
\enlargethispage{2\baselineskip}

\end{document}